\documentclass[11pt]{amsart}
\usepackage{graphicx}
\usepackage{amsfonts}
\usepackage{amssymb}
\usepackage{amsthm}
\usepackage[centertags]{amsmath}
\usepackage{newlfont}

\usepackage{color}
\usepackage{graphicx,color}
\usepackage{amsmath, amssymb, graphics}
 
\def\r{\mathbb R}

\newtheorem{theorem}{Theorem}[section]
\newtheorem{proposition}[theorem]{Proposition}

\newtheorem{lemma}[theorem]{Lemma}

\theoremstyle{definition}

\theoremstyle{remark}
\newtheorem{remark}[theorem]{Remark}

\numberwithin{equation}{section}

\begin{document}

 \title{A characteristic property of Delaunay surfaces}

\author{Thomas Hasanis}
\address{$^1$ Department of Mathematics\\
               University of Ioannina\\
               45110 Ioannina, Greece}
\email{thasanis@cc.uoi.gr}

\author{Rafael L\'opez}
\address{$^2$ Departamento de Geometr\'{\i}a y Topolog\'{\i}a\\
 Universidad de Granada\\
 18071 Granada, Spain }
\email{rcamino@ugr.es}
\thanks{Rafael L\'opez has  partially supported by  the grant no. MTM2017-89677-P, MINECO/AEI/FEDER, UE.}

\subjclass{Primary 53A10; Secondary 53C42}

\keywords{mean curvature, separable surface, Delaunay surface}
\date{}

\begin{abstract} 
    We prove that   Delaunay surfaces, except the plane and the catenoid, are the only surfaces  in Euclidean space with nonzero constant mean curvature    that can be expressed as an implicit equation of type $f(x)+g(y)+h(z)=0$, where $f$, $g$ and $h$ are smooth real functions of one variable.
\end{abstract}
 
\maketitle

\section{Introduction and statement of the result} 

We study surfaces in the Euclidean $3$-space $\r^3$ that can be expressed by an implicit equation of type 
\begin{equation}\label{sep}
 f(x)+g(y)+h(z)=0,
 \end{equation}
  where $f$, $g$ and $h$ are smooth functions of one variable. Here $(x,y,z)$ stand for the canonical coordinates of $\r^3$. In order to fix the terminology, a surface of $\r^3$ given by the implicit equation \eqref{sep} is called a {\it separable surface}.

 This paper is motivated by  the classical theory of minimal surfaces, where the study of the class of  separable minimal surfaces  has a long history and that we now summarize. First of all, the first two minimal surfaces discovered are separable surfaces, namely, the catenoid   $\cosh^2(z)=x^2+y^2$   by Euler in 1744, and    the helicoid $\tan(z)=y/x$   by Meusnier 1776.

Later, in 1835,   Scherk found new examples of minimal surfaces by using the technique of separation of variables (\cite{sc}). Among the examples, we point out the separable surfaces $e^z=\cos(y)/\cos(x)$ and $\sin(z)=\sinh(x)\sinh(y)$. In 1887,  Weingarten investigated the problem to determine all separable minimal surfaces realizing that form a rich and large family of surfaces (\cite{we}).  For example, this family contains a variety of  minimal surfaces given in terms of elliptic integrals (\cite{ca}) as well as  periodic minimal surfaces such as the    P-Schwarz surface and the D-Schwarz surface  (\cite{schwarz}).   In the middle of the above century,   Fr\'echet  gave a deep study of these surfaces obtaining   examples with explicit parametrizations (\cite{fr1,fr2}).   The   reader can see a historical approach to the topic of separable minimal surfaces in  the Nitsche's book   (\cite[II-5.2]{ni}).

If the class of separable minimal surfaces  has long been investigated by many authors,  the question of what are the separable surfaces with nonzero constant  mean curvature     has not been addressed nowdays in the literature. This paper covers this case and gives a complete classification of these surfaces in the following result.

\begin{theorem}\label{t1}
The Delaunay surfaces, except the plane and the catenoid, are the only separable surfaces in    Euclidean  space $\r^3$ with nonzero constant mean curvature  
 \end{theorem}

The Delaunay surfaces are the   surfaces of revolution  with constant mean curvature. In 1841,  Delaunay  provided a geometric method of construction of these surfaces (\cite{de}).  He proved that a rotational surface with constant mean curvature is obtained by rolling a conic   without slipping along a straight line. The curve described by the focus of the conic, called the roulette of the conic, is the generating curve of the rotational surface and the straight line is the axis of rotation. There are three types of Delaunay surfaces depending on the type of conic. 
  If the conic is a ellipse, then the surface is called an unduloid (including spheres and circular cylinders as special cases), if the conic is a hyperbola, the surface is called a nodoid and if the conic is a parabola, the   surface is a catenoid. Besides the plane, the catenoid is the only rotational minimal surface. 

 This note is organized as follows. In Section \ref{s2} we compute the mean curvature of a separable surface \eqref{sep} and we distinguish   the surfaces of revolution and the translation surfaces as special  cases in the proof of Theorem \ref{t1}. In  the last Section \ref{s3}, we prove Theorem \ref{t1}.

\section{Preliminaries}\label{s2}


Let $S$ be a  a separable surface   of the form \eqref{sep}, where $f$, $g$ and $h$ are   smooth functions defined in open intervals $I_1$, $I_2$ and $I_3$ of $\r$, respectively. 
 By the regularity of $S$,    $f'(x)^2+g'(y)^2+h'(z)^2>0$ for every $x\in I_1$, $y\in I_2$ and $z\in I_3$.  
 
 There are  two   special examples of separable surfaces that appear   by simples choices of the functions $f$, $g$ and $h$ in   \eqref{sep}.
 
 \begin{enumerate}

 \item {\it Rotational surfaces}. Let $S$ be a   rotational surface  in $\r^3$ whose rotation axis is one of the coordinate axes,   for instance,    the   $z$-axis. Then the surface $S$ writes as $h(z)=x^2+y^2$, hence that   $S$   is   separable.   A little more general, any surface of revolution about a straight line parallel to the $z$-axis is separable because its implicit equation is   $h(z)=x^2+y^2+ax+by+c$, with $a,b,c\in\r$.

\item {\it Translation surfaces}.   After renaming the coordinates $(x,y,z)$,   a translation surface $S$ is a surface   can be expressed as $z=\phi(x)+\psi(y)$, where $\phi$ and $\psi$ are  smooth functions.  The surface $S$ can be viewed  as   the sum of the plane curve $x\mapsto (x,0,\phi(x))$  and the plane curve $y\mapsto (0,y,\psi(y))$. Thus $S$ is generated when we move one of  these curves by means of the translations along the other one.    A separable surface \eqref{sep} is  a    translation surface  $z=\phi(x)+\psi(y)$  if and only the function $h$ is linear,  say,   $h(z)=az+b$, with $a,b\in\r$. In case   $a=0$, the surface is a right cylinder over a plane curve contained in the $xy$-plane.  
  
 \end{enumerate}

We calculate the mean curvature $H$ of  a separable surface. Firstly, we recall that any surface  $S$ is, locally, the zero level set $F(x,y,z)=0$ of a function $F$ defined in an open set $O\subset\r^3$. For computing the mean curvature $H$ of $S$ we fix an orientation on $S$, which here it is   $N=\nabla F/|\nabla F|$, where $\nabla$ is the Euclidean gradient in $\r^3$. With respect to $N$, the mean curvature is 
\begin{equation}\label{mean}
H=-\frac12\mbox{Div}(N)=-\frac12\mbox{Div}\left(\frac{\nabla F}{|\nabla F|}\right),
\end{equation}
where $\mbox{Div}$ is the Euclidean  divergence   in $\r^3$. 
If now the surface is separable given by \eqref{sep}, we have $\nabla F=(f',g',h')$. Here the prime (') stands for the derivative of the function with respect to its variable. Then the formula \eqref{mean} is now 
\begin{eqnarray*}-2H&=&  \left(\frac{f'}{\sqrt{f'^2+g'^2+h'^2}}\right)_x+\left(\frac{g'}{\sqrt{f'^2+g'^2+h'^2}}\right)_y+\left(\frac{h'}{\sqrt{f'^2+g'^2+h'^2}}\right)_z\\\
&=&\frac{1}{ (f'^2+g'^2+h'^2)^{3/2}}\left(f''(g'^2+h'^2)+g''(f'^2+h'^2)+h''(f'^2+g'^2)\right).
\end{eqnarray*}
 We write this equation as 
\begin{equation}\label{eq1}
f''(g'^2+h'^2)+g''(f'^2+h'^2)+h''(f'^2+g'^2)=-2H(f'^2+g'^2+h'^2)^{3/2}.
\end{equation}

We know that the  rotational surfaces with constant mean curvature are the Delaunay surfaces. If the $z$-axis is the rotation axis, then $h(z)=x^2+y^2$ and \eqref{eq1} is     simply 
$$-2H= \frac{8h+4h'^2-4hh''}{(4h+h'^2)^{3/2}}.$$
Multiplying by $h'$, we derive a first integral, namely, 
$$\frac{2h}{\sqrt{4h+h'^2}}=-Hh+c,$$
where $c\in\r$ is a constant of integration. This equation characterizes all Delaunay surfaces. For example, the catenoid ($H=0$) is $h(z)=\cosh^2(z)$, the sphere is $h(z)=1/H^2-z^2$ and the cylinder is $h(z)=1/(4H^2)$. 

On the other hand, the translation surfaces $z=\varphi(x)+\psi(y)$ with constant  mean curvature $H$ are known. If $H=0$, the surface is the plane of the   Scherk surface $e^z=\cos(y)/\cos(x)$, or equivalently, $z=\log(\cos(y))-\log(\cos(x))$.   If $H\not=0$,   then \eqref{eq1} is now 
$$(1+\psi'^2)\phi''+(1+\phi'^2)\psi''=-2H(1+\phi'^2+\psi'^2)^{3/2}.$$
Liu in \cite{li} proved that then $\psi$ (or $\phi)$ is a linear function, hence $(1+a^2)\phi''=-2H(1+a^2+\phi'^2)^{3/2}$ for some constant $a\in\r$ and the surface is  
\begin{equation}\label{cyl}
z=\frac{\sqrt{1+a^2}}{2H}\sqrt{1-4H^2 x^2}+a y,\quad a\in\r.
\end{equation}
This surface is nothing but a circular cylinder of radius $1/(2|H|)$ whose rotation axis is the $y$-axis once rotated in the $yz$-plane an angle $\theta$ with $\tan(\theta)=a$.     In particular, a circular cylinder is  a Delaunay surface.

In the proof of Theorem \ref{t1} that will be done in Section \ref{s3}, we will not use the functions $f$, $g$ and $h$ such as appear in \eqref{eq1}, but we will make a change of variables. For this purpose, firstly we need  to distinguish the case that  one of the functions $f$, $g$ or $h$ is constant. Without loss of generality, we suppose that $h$ is constant, so $h(z)=a$ for some $a\in\r$. By \eqref{sep},  the surface is given by the implicit equation $f(x)+g(y)+a=0$ and this shows that the surface is a right cylinder over the plane curve $C=\{(x,y)\in\r^2: f(x)+g(y)+a=0\}$. Then the mean curvature of $S$ is $H=\kappa/2$, where $\kappa$ is the curvature of $C$. If $H=0$ is a nonzero constant, then $\kappa=0$, $C$ is a straight line and $S$ is a plane. If    $H$ is a nonzero constant, then  $C$ is (part of) a circle of radius $r=1/(2|H|)$ and  $S$ is a circular cylinder. In this last case,  Theorem \ref{t1} is proved.

 In   what follows,  we will assume that none of the functions $f$, $g$ or $h$ is constant.  Equivalently, we have
\begin{equation}\label{fgh}
f'(x)g'(y)h'(z)\not=0
\end{equation}
everywhere in $\Omega=I_1\times I_2\times I_3$.  This allows to define   the new variables 
\begin{equation}\label{uvw}
u=f(x),\quad v=g(y),\quad w=h(z),
\end{equation}
which are related by the equation $u+v+w=0$ thanks to \eqref{sep}.

Let us introduce the  functions
$$X(u)=f'(x)^2,\quad Y(v)=g'(y)^2,\quad Z(w)=h'(z)^2.$$
Using \eqref{fgh}, we have   
$$X'(u)= 2f''(x), \quad Y'(v) =2g''(y), \quad Z'(w) =2h''(z).$$
 Thus  the equation \eqref{eq1} becomes
\begin{equation}\label{eq2}
(Y+Z)X'+(X+Z)Y'+(X+Y)Z'=-4H(X+Y+Z)^{3/2}.
\end{equation}
Throughout this paper we need to differentiate equations similar to \eqref{eq2} involving functions depending on $u$, $v$ and $w$. Since these variables are not independent because of $u+v+w=0$, the following lemma will be useful in our computations.

\begin{lemma} \label{le1}
Let $A=A(u,v,w)$ be a smooth function defined in a domain $O\subset\r^3$. If $A(u,v,w)=0$ for any triple of the section $O\cap\Pi$, where $\Pi$ is the plane of equation $u+v+w=0$, then on the section we have
$$A_u=A_v=A_w,$$
 where $A_u$, $A_v$ and $A_w$ are the derivatives of $A$ with respect to $u$, $v$ and $w$, respectively.
\end{lemma}
\begin{proof}
Since $w=-u-v$, then $A(u,v,-u-v)=0$. Differentiating with respect to $u$, we deduce $A_u-A_w=0$.   Changing the roles of $u$, $v$ and $w$, we conclude the proof of the lemma. 
 \end{proof}

In the following result, we characterize  in terms of the functions $X$, $Y$ and $Z$ the particular cases in which  the surface is a translation or a rotational surface. 
  
\begin{proposition}\label{pr1}
 Let $S$ be a   separable  surface with $(X'-Y')(Y'-Z')(Z'-X')=0$ everywhere. Then $S$ is a rotational  or a translation surface.
\end{proposition}

\begin{proof}
Without loss of generality, we suppose that $X'-Y'=0$ everywhere. Then there is a constant $a\in\r$ such that $X'=a=Y'$. If $a=0$, then $X$ and $Y$ are nonzero constants because \eqref{fgh}. So $f(x)=a_1x+b_1$ and $g(y)=a_2y+b_2$, where $a_i,b_i\in\r$ and, by using \eqref{fgh}, $a_i\not=0$. Then \eqref{sep} is $a_1x+a_2y+h(z)+b_1+b_2=0$ and this proves that $S$ is a translation surface. 

Assume now $a\not=0$. Then the solution of $X'=a=Y'$ yields
$$f'(x)^2=a f(x)+b,\quad g'(y)^2=ag(y)+b_2,$$
for some $b_1,b_2\in\r$. Then solutions of these ordinary differential equations are 
$$f(x)=\frac{(ax+c_1)^2}{4a}-\frac{b_1}{a},\quad g(y)=\frac{(ax+c_2)^2}{4a}-\frac{b_2}{a},$$
where $c_1,c_2\in\r$. Hence  the surface is a surface of revolution    with respect to a straight line parallel to  the $z$-axis.
\end{proof}

\section{Proof of Theorem \ref{t1}}\label{s3}

In this section we prove Theorem \ref{t1}. Let $S$ be a surface with constant mean curvature $H\not=0$ and given by the implicit equation \eqref{sep}. By the previous section, we know that if $S$ is a translation surface, then $S$ is a circular cylinder, in particular, a Delaunay surface. 

Assume now that $S$ is not a translation surface. The proof of Theorem \ref{t1} is by contradiction. Suppose that $S$ is not a surface of revolution. We use the notation of Section \ref{s2}. Because $S$ is not a right cylinder,  we know that $f'g'h'\not=0$ everywhere on $\Omega=I_1\times I_2\times I_3$. Let us take the new variables $u$, $v$ and $w$ defined in \eqref{uvw}.

Since $S$ is not a translation surface neither a rotational surface,  we have
\begin{equation}\label{31}
(X'-Y')(Y'-Z')(Z'-X')\not=0
\end{equation}
in an open set of $\Omega$ by Proposition \ref{pr1}.

We write equation \eqref{eq2} as
\begin{equation}\label{eq3}
(X+Y)Z'+(X'+Y')Z+X'Y+XY'=-4H(X+Y+Z)^{3/2}.
\end{equation}
Applying Lemma \ref{le1} to this equation for $u$ and $v$, we obtain
$$
(X'-Y')Z'+(X''-Y'')Z+X''Y-XY''=-6H(X'-Y')(X+Y+Z)^{1/2}.
$$
By using this equation and equation   \eqref{eq3}, we eliminate $Z'$  and we find
\begin{eqnarray*}
&&\left(X'^2-Y'^2-(X+Y)(X''-Y'')\right)Z+(X'-Y')(X'Y+XY')\\
&&-(X+Y)(X''Y-XY'')\\
&&=2H(X'-Y')(X+Y-2Z)(X+Y+Z)^{1/2}
\end{eqnarray*}
or 
\begin{equation}\label{pz}
PZ+Q=2H(X'-Y')(X+Y-2Z)(X+Y+Z)^{1/2},
\end{equation}
where we have set
\begin{eqnarray*}
P&=&X'^2-Y'^2-(X+Y)(X''-Y'')\\
Q&=&(X'-Y')(X'Y+XY')-(X+Y)(X''Y-XY'').
\end{eqnarray*}
The squared of equation \eqref{pz} gives 
\begin{eqnarray*}
&&16H^2(X'-Y')^2Z^3-P^2Z^2-\big(12H^2(X+Y)^2(X'-Y')^2+2PQ\big)Z\\
&&+4H^2(X+Y)^3(X'-Y')^2-Q^2=0
\end{eqnarray*}
or 
\begin{equation}\label{eq6}
LZ^3+MZ^2+NZ+R=0,
\end{equation}
where for convenience, we have set
\begin{eqnarray*}
L&=&16H^2(X'-Y')^2\\
 M&=&-P^2\\
 N&=&- 12H^2(X+Y)^2(X'-Y')^2-2PQ\\
 R&=&4H^2(X+Y)^3(X'-Y')^2-Q^2
\end{eqnarray*}

We write equation \eqref{eq6} as
\begin{equation}\label{z3}
Z^3+\frac{M}{L}Z^2+\frac{N}{L}Z+\frac{R}{L}=0
\end{equation}
because $L\not=0$ by \eqref{31}. This is a polynomial equation in $Z$ and thus its coefficients depend on the roots $Z_1(w)$, $Z_2(w)$, $Z_3(w)$. Hence, these coefficients  are functions of one variable $w=-(u+v)$. Set 
\begin{eqnarray*}
&&\Phi_1(u+v)=\frac{M(u,v)}{L(u,v)}\\
&&\Phi_2(w+v)=\frac{N(u,v)}{L(u,v)}\\
&&\Phi_3(u+v)=\frac{R(u,v)}{L(u,v)}.
\end{eqnarray*}
With this notation, we have from \eqref{z3} that 
$$
Z^3+\Phi_1 Z^2+\Phi_2 Z+\Phi_3=0.
$$
We apply Lemma \ref{le1} to this equation for $w$, $u$, obtaining 
\begin{equation}\label{eq8}
(3Z^2+2\Phi_1 Z+\Phi_2)Z'=\Phi_1'Z^2+\Phi_2' Z+\Phi_3'.
\end{equation} 
From \eqref{eq3}, we get $Z'$, namely,  
$$Z'=-\frac{1}{X+Y}\left((X'+Y')Z+X'Y+XY'+4H(X+Y+Z)^{3/2}\right)$$
and we insert this expression for $Z'$ in \eqref{eq8}. It follows that   
\begin{eqnarray*}
&&(3Z^2+2\Phi_1 Z+\Phi_2)\Big((X'+Y')Z+X'Y+XY'+4H(X+Y+Z)^{3/2}\Big)\\
=&&-(X+Y)(\Phi_1'Z^2+\Phi_2'Z+\Phi_3')
\end{eqnarray*}
or equivalently
\begin{eqnarray*}
&&(3Z^2+2\Phi_1Z+\Phi_2)\Big((X'+Y')Z+X'Y+XY'\Big)\\
&&+(X+Y)\Big(\Phi_1' Z^2+\Phi_2' Z+\Phi_3'\Big)\\
&=&-4H(3Z^2+2\Phi_1 Z+\Phi_2)(X+Y+Z)^{3/2}.
\end{eqnarray*}
The squared of this equation gives 
\begin{eqnarray*}
&&(3Z^2+2\Phi_1 Z+\Phi_2)^2\Big( (X'+Y')Z+X'Y+XY'\Big)^2\\
&&+(X+Y)^2(\Phi_1'Z^2+\Phi_2'Z+\Phi_3')^2\\
&&+2(X+Y)\Big((X'+Y')Z+X'Y+XY'\Big)(3Z^2+2\Phi_1Z+\Phi_2)(\Phi_1' Z^2+\Phi_2'Z+\Phi_3')\\
&&=16 H^2(3Z^2+2\Phi_1 Z+\Phi_2)^2(X+Y+Z)^3.
\end{eqnarray*}
 This equation is a polynomial equation in $Z$ of degree $7$, that is, $\sum_{n=0}^7 a_n Z^n=0$. The coefficient $a_7$ is $a_7=144H^2$ a nonzero constant. So, all other coefficients depend on  the roots of   equation. Thus these  are functions of   $w=-(u+v)$. The computation of the coefficient $a_6$ leads to  
$$  a_6=16H^2(27(X+Y)+12\Phi_1)-9(X'+Y')^2.$$
 We apply Lemma \ref{le1} to the above equation differentiating with respect to $u$ and $v$. Then $(a_6)_u-(a_6)_v=0$ is 
   \begin{equation}\label{eq10}
 24H^2(X'-Y')-(X'+Y')(X''-Y'')=0.
 \end{equation}
 Notice that  $X'+Y'\not=0$, otherwise we have from \eqref{eq10} that  $X'-Y'=0$,  a contradiction by \eqref{31}.  Hence 
 \begin{equation}\label{eq11}
 24H^2\frac{X'-Y'}{X'+Y'}=X''-Y''.
 \end{equation}
This equation depends on the variables $u$ and $v$, which are  independent in this equation. Thus, the derivative of \eqref{eq11} with respect to $u$ is  
 $$48 H^2\frac{X''Y'}{(X'+Y')^2}=X''',$$
 and the derivative of this equation with respect to $v$ yields
  $$48H^2(X'-Y')\frac{X''Y''}{(X'+Y')^3}=0. $$
 From this equation and \eqref{31}, we deduce $X''Y''=0$.  
 
 In the arguments after equation \eqref{eq3}, we change the roles of $X$ and $Y$  by $Y$ and $Z$ and, finally by $X$ and $Z$.  In a similar way, we find $X''Z''=Y''Z''=0$, and so 
$$ X''Y''=Y''Z''=Z''X''=0.$$
 It follows that   at least two of the functions $X''$, $Y''$, $Z''$ must be zero. If, for instance, $X''=Y''=0$, then   equation \eqref{eq10} implies  
 $24H^2(X'-Y')=0$, which is a contradiction by \eqref{31} and   because $H\not=0$. 
 
 For the other cases, namely, $Y''=Z''=0$ and $X''=Z''=0$, we use the analogous relation of \eqref{eq10}, obtaining in each case a contradiction. This completes the  proof of Theorem \ref{t1}.

\begin{remark} From the proof of Theorem \ref{t1} we conclude that the rotation axis is parallel to one of the coordinate axes or, in case that the axis is not parallel to the coordinate axes, then the surface is a circular cylinder. This case appears when the surface is of translation type and the axis of the circular cylinder is arbitrary: see equation \eqref{cyl}.
\end{remark}

\bibliographystyle{amsplain}

\end{document}